\documentclass[12pt]{article}

\usepackage{amsmath,amssymb,amsfonts,amsthm,graphics,verbatim}

\newtheorem{theorem}{Theorem}
\newtheorem{lemma}[theorem]{Lemma}

\newtheorem{proposition}[theorem]{Proposition}
\newtheorem{bc}[theorem]{Brown's Filtration Criterion}

\def\slp{{\mathbf{SL_n}(\mathbb{Z}[t])}}
\def\sll{{\mathbf{SL_n}(\mathbb{Q}((t^{-1})))}}
\def\slr{{\mathbf{SL_n}(\mathbb{Q}[t])}}
\def\slt{{\mathbf{SL_n}(\mathbb{Q}[[t^{-1}]])}}

\begin{document}

\title{ $\slp$ is not $FP_{n-1}$}
\author{Kai-Uwe Bux, Amir Mohammadi, Kevin Wortman
\thanks{Supported in part by an N.S.F. Grant DMS-0604885.} }
\date{December 12, 2007}
\maketitle

\begin{abstract}
We prove the result from the title using the geometry of Euclidean
buildings.
\end{abstract}

\section{Introduction}\label{s:intro}

Little is known about the finiteness properties of $\slp$ for
arbitrary $n$.

In 1959 Nagao proved that if $k$ is a field then
$\mathbf{SL_2}(k[t])$ is a free product with amalgamation
\cite{Nagao}. It follows from his description that
$\mathbf{SL_2}(\mathbb{Z}[t])$ and its abelianization are not
finitely generated.

In 1977 Suslin proved that when $n\geq 3,$ $\slp$ is finitely
generated by elementary matrices \cite{Suslin}. It follows that
$H_1(\slp , \mathbb{Z})$ is trivial when $n \geq 3$.

More recent, Krsti\'{c}-McCool proved that
$\mathbf{SL_3}(\mathbb{Z}[t])$ is not finitely presented
\cite{K-M}.

It's also worth pointing out that since $\slp$ surjects onto
$\mathbf{SL_n}(\mathbb{Z})$, that $\slp$ has finite index
torsion-free subgroups.

In this paper we provide a generalization of the results of Nagao
and Krsti\'{c}-McCool mentioned above for the groups $\slp$.

\begin{theorem}\label{t:main}
If $n\geq 2$, then $\slp$ is not of type $FP_{n-1}$.
\end{theorem}

Recall that a group $\Gamma$ is \emph{of type }$FP_m$ if if there
exists a projective resolution of $\mathbb{Z}$ as the trivial
$\mathbb{Z}\Gamma $ module
$$P_m \rightarrow P_{m-1} \rightarrow \cdot \cdot \cdot \rightarrow P_1 \rightarrow P_0 \rightarrow \mathbb{Z} \rightarrow 0$$
where each $P_i$ is a finitely generated $\mathbb{Z} \Gamma$
module.

In particular, Theorem~\ref{t:main} implies that there is no
$K(\slp , 1)$ with finite $(n-1)$-skeleton, where $K(G,1)$ is the
Eilenberg-Mac\,Lane space for $G$.

\begin{subsection}{Outline of paper}
The general outline of this paper is modelled on the proofs in
\cite{BW1} and \cite{BW2}, though some important modifications
have to be made to carry out the proof in this setting.

As in \cite{BW1} and \cite{BW2}, our approach is to apply Brown's
filtration criterion \cite{Brown filtration}. Here we will examine
the action of $\slp$ on the locally infinite Euclidean building
for $\sll$. In Section~\ref{s:stab} we will show that the infinite
groups that arise as cell stabilizers for this action are of type
$FP_m$ for all $m$, which is a technical condition that is needed
for our application of Brown's criterion.

In Section~\ref{s:tori} we will demonstrate the existence of a
family of diagonal matrices that will imply the existence of a ``nice"
isometrically embedded codimension $1$ Euclidean space in the
building for $\sll$. In \cite{BW1} analogous families of diagonal
matrices were constructed using some standard results from the
theory of algebraic groups over locally compact fields. Because
$\mathbb{Q}((t^{-1}))$ is not locally compact, our treatment in
Section~\ref{s:tori} is quite a bit more hands on.

Section~\ref{s:body} contains the main body of our proof. We use
translates of portions of the codimension 1 Euclidean subspace
found in Section~\ref{s:tori} to construct spheres in the
Euclidean building for $\sll$ (also of codimension 1). These
spheres will lie ``near'' an orbit of $\slp$, but will be nonzero
in the homology of cells ``not as near'' the same $\slp$ orbit.
Theorem~\ref{t:main} will then follow from Brown's criterion.

\end{subsection}

\section{Stabilizers}\label{s:stab}

\begin{lemma}\label{l:stab} If $X$ is the Euclidean building for $\sll$, then the $\slp$ stabilizers of cells in $X$ are
$FP_m$ for all $m$. \end{lemma}

\begin{proof} Let $x_0 \in X$ be the vertex stabilized by $\slt$. We
denote a diagonal matrix in $\mathbf{GL_n}(\mathbb{Q}((t^{-1})))$
with entries $s_1, s_2,..., s_n \in \mathbb{Q}((t^{-1}))^\times$
by $D(s_1,s_2,...,s_n)$, and we let $\mathfrak{S} \subseteq X$ be
the sector based at $x_0$ and containing vertices of the form
$D(t^{m_1}, t^{m_2},...,t^{m_n})x_0$ where each $m_i \in
\mathbb{Z}$ and $m_1 \geq m_2 \geq ... \geq m_n$.

The sector $\mathfrak{S}$ is a fundamental domain for the action
of $\slr$ on $X$ (see \cite{So}). In particular, for any vertex $z
\in X$, there is some $h'_z\in \slr$ and some integers $m_1 \geq
m_2 \geq ... \geq m_n$ with $z=h'_zD_z(t^{m_1},
t^{m_2},...,t^{m_n})x_0$. We let $h_z=h'_zD_z(t^{m_1},
t^{m_2},...,t^{m_n}).$

For any $N \in \mathbb{N}$, let $W_N$ be the $(N+1)$-dimensional
vector space $$W_N=\{\,p(t) \in \mathbb{C}[t] \mid
\text{deg}\big(p(t)\big) \leq N \}$$ which is endowed with the
obvious $\mathbb{Q}-$structure. If $N_1,\cdots,N_{n^2}$ in
$\mathbb{N}$ are arbitrary then let
$$\mathbf{G}_{\{N_1,\cdots,N_{n^2}\}}=\{\mathbf{x}\in\prod_{i=1}^{n^2}W_{N_i}|
\hspace{1mm}\mbox{det}(\mathbf{x})=1\}$$ where
$\mbox{det}(\mathbf{x})$ is a polynomial in the coordinates of
$\mathbf{x}.$ To be more precise this is obtained from the usual
determinant function when one considers the usual $n\times n$
matrix presentation of $\mathbf{x},$ and calculates the
determinant in $\mathbf{Mat}_n(\mathbb{C}[t]).$

For our choice of vertex $z \in X$ above, the stabilizer of $z$ in
$\sll$ equals $h_z \slt h_z^{-1}.$ And with our fixed choice of
$h_z$, there clearly exist some $N^z_i \in \mathbb{N}$ such that
the stabilizer of the vertex $z$ in $\slr$ is
$\mathbf{G}_{\{N^z_1,\cdots,N^z_{n^2}\}}(\mathbb{Q})$.
Furthermore, conditions on $N^z_i$ force a group structure on
$\mathbf{G}_z=\mathbf{G}_{\{N^z_1,\cdots,N^z_{n^2}\}}.$ Therefore,
the stabilizer of $z$ in $\slr$ is the $\mathbb{Q}-$points of the
affine $\mathbb{Q}$-group $\mathbf{G}_z$, and the stabilizer of
$z$ in $\slp$ is $\mathbf{G}_z(\mathbb{Z})$.

The action of $\slr$ on $X$ is type preserving, so if $\sigma
\subset \mathfrak{S}$ is a simplex with vertices $z_1, z_2,...,z_m$, then
the stabilizer of $\sigma$ in $\slp$ is simply
$$\big(\mathbf{G}_{z_1} \cap ... \cap
\mathbf{G}_{z_m}\big)(\mathbb{Z})$$ That is, the stabilizer of
$\sigma$ in $\slp$ is an arithmetic group, and Borel-Serre proved
that any such group is $FP_m$ for all $m$ \cite{B-S}.

 \end{proof}

\section{Polynomial points of tori}\label{s:tori}

This section is devoted exclusively to a proof of the following

\begin{proposition}\label{p:tori} There is a group $A \leq \slp$ such that
\begin{quote} (i) $A \cong \mathbb{Z}^{n-1}$ \\
(ii) There is some $g \in \sll$ such that $gAg^{-1}$ is a group of
diagonal matrices \\
(iii) No nontrivial element of $A$ fixes a point in the Euclidean
building for $ \sll $.
\end{quote}
\end{proposition}

The proof of this proposition is modelled on a classical approach
to finding diagonalizable subgroups of
$\mathbf{SL_n}(\mathbb{Z})$. The proof will take a few steps.

\begin{subsection}{A polynomial over $\mathbb{Z}[t]$ with roots
in $\mathbb{Q}((t^{-1}))$}\label{s:poly}

Let $\{p_1,p_2,p_3, ...\}=\{2,3,5, ...\}$ be the sequence of prime
numbers. Let $q_1 =1$. For $2\leq i \leq n$, let $q_i = p_{i-1}+1$.

Let $f(x) \in \mathbb{Z}[t][x]$ be the polynomial given by
$$f(x)=\Big[\prod_{i=1}^n(x+q_it)\Big]-1$$ It will be clear by
the conclusion of our proof that $f(x)$ is irreducible over
$\mathbb{Q}(t)$, but we will not need to use this directly.

\begin{lemma}\label{l:root} There is some $\alpha \in \mathbb{Q}((t^{-1}))$
 such that $f(\alpha)=0$. \end{lemma}

\begin{proof} We want to show that there are $c_i \in \mathbb{Q}$
such that if $\alpha = \sum_{i=0}^\infty c_i t^{1-in}$ then
$f(\alpha)=0$.\vspace{.75mm}

To begin let $c_0 =-1$. We will define the remaining $c_i$
recursively. Define $c_{i,k}$ by $\alpha +q_kt = \sum_{i=0}^\infty
c_{i,k} t^{1-in}$. Thus, $c_{i,k}=c_i$ when $i\geq 1$, each
$c_{0,k}$ is contained in $\mathbb{Q}$, and $c_{0,1}=0$.

That $\alpha$ is a root of $f$ is equivalent to \begin{align*} 1
&= \prod_{k=1}^n (\alpha +q_kt)= \prod_{k=1}^n \big(\sum_{i=0}^\infty c_{i,k} t^{1-in}\big) \\
&=\sum_{i=0}^\infty \big( \sum _{\sum_{k=1}^ni_k=i} \big(\prod
_{k=1}^n c_{i_k,k}\big) \big) t^{n(1-i)}
\end{align*}

Our task is to find $c_m$'s so that the above is satisfied.

Note that for the above equation to hold we must have
$$0\cdot t^n =  \sum _{\sum_{k=1}^ni_k=0} \big(\prod
_{k=1}^n c_{i_k,k}\big) t^{n(1-0)}$$ That is $$0 = \prod_{k=1}^n
c_{0,k}$$ which is an equation we know is satisfied because
$c_{0,1}=0$. Now assume that we have determined
$c_0,c_1,...,c_{m-1} \in \mathbb{Q}$. We will find $c_m \in
\mathbb{Q}$.

Notice that the first coefficient in our Laurent series expansion
above which involves $c_m$ is the coefficient for the $t^{-nm}$
term. This follows from the fact that each $i_k$ is nonnegative.

Since $$\sum _{\sum_{k=1}^ni_k=m} \big(\prod _{k=1}^n
c_{i_k,k}\big)$$ is the coefficient of the $t^{-nm}$ term in
the expansion of $1$, we have $$0=\sum _{\sum_{k=1}^ni_k=m}
\big(\prod _{k=1}^n c_{i_k,k}\big)$$

The above equation is linear over $\mathbb{Q}$ in the single
variable $c_m$ and the coefficient of $c_m$ is nonzero. Indeed, $\sum_{k=1}^ni_k=m$, each $i_k \geq 0$, and $c_0,...,c_{m-1}\in \mathbb{Q}$ are assumed to be known
quantities. Thus, $c_m \in \mathbb{Q}$.

\end{proof}

\end{subsection}

\begin{subsection}{Matrices representing ring
multiplication}\label{s:mat}

By Lemma~\ref{l:root} we have that the field
$\mathbb{Q}(t)(\alpha) \leq \mathbb{Q}((t^{-1}))$ is an extension
of $\mathbb{Q}(t)$ of degree $d$ where $d\leq n$. It follows that
$\mathbb{Z}[t][\alpha]$ is a free $\mathbb{Z}[t]$-module of rank
$d$ with basis $\{1,\alpha , \alpha ^2 , ... , \alpha ^{d-1} \}$.

For any $y \in \mathbb{Z}[t][\alpha]$, the action of $y$ on
$\mathbb{Q}(t)(\alpha )$  by multiplication is a linear
transformation that stabilizes $\mathbb{Z}[t][\alpha]$. Thus, we
have a representation of $\mathbb{Z}[t][\alpha]$ into the ring of
$d \times d$ matrices with entries in $\mathbb{Z}[t]$. We embed
the ring of $d\times d$ matrices with entries in $\mathbb{Z}[t]$
into the upper left corner of the ring of $n \times n$ matrices
with entries in $\mathbb{Z}[t]$.

By Lemma~\ref{l:root} $$\prod_{i=1}^n(\alpha + q_it)=1$$ so each
of the following matrices are invertible:
$$\alpha + q_1t,\; \alpha + q_2t,\;  ...,\; \alpha + q_nt$$ (We
will be blurring the distinction between the elements of
$\mathbb{Z}[t][\alpha]$ and the matrices that represent them.)

For $1\leq i \leq n-1$, we let $a_i =\alpha + q_{i+1}t$. Since
$a_i$ is invertible, it is an element of
$\mathbf{GL_n}(\mathbb{Z}[t])$, and hence has determinate $\pm1$.
By replacing each $a_i$ with its square, we may assume that $a_i
\in \slp$ for all $i$. We let $A=\langle a_1,...a_{n-1} \rangle$
so that $A$ is clearly abelian as it is a representation of
multiplication in an integral domain. This group $A$ will satisfy
Proposition~\ref{p:tori}.

\end{subsection}

\begin{subsection}{$A$ is free abelian on the $a_i$}

To prove part (i) of Proposition~\ref{p:tori} we have to show
that if there are $m_i \in \mathbb{Z}$ with
$$\prod_{i=1}^{n-1}a_i^{m_i}=1$$ then each $m_i=0$. But the first
nonzero term in the Laurent series expansion for $\alpha$ is $-t$,
which implies that the first nonzero term in the Laurent series
expansion for each $a_i$ is $-t+q_{i+1}t=p_it$. Hence, the first
nonzero term of
$$\prod_{i=1}^{n-1}a_i^{m_i}=1$$ is
$$\prod_{i=1}^{n-1}(p_it)^{m_i}=t^0$$ Thus
$$\prod_{i=1}^{n-1}p_i^{m_i}=1$$and it follows by the uniqueness
of prime factorization that $m_i=0$ for all $i$ as desired.

Thus, part (i) of Proposition~\ref{p:tori} is proved.

\end{subsection}

\begin{subsection}{$A$ is diagonalizable}

Recall that $\alpha$ is a $d\times d$ matrix with entries in
$\mathbb{Z}[t]$ where $d$ is the degree of the minimal polynomial
of $\alpha$ over $\mathbb{Q}(t)$. Let that minimal polynomial be
$q(x)$. Because the characteristic of $\mathbb{Q}(t)$ equals $0$,
$q(x)$ has distinct roots in $\mathbb{Q}(t)(\alpha)$.

Let $Q(x)$ be the characteristic polynomial of the matrix
$\alpha$. The polynomial $Q$ also has degree $d$ and leading
coefficient $\pm 1$ with $Q(\alpha)=0$. Therefore, $q=\pm Q$.
Hence, $Q$ has distinct roots in $\mathbb{Q}(t)(\alpha )$ which
implies that $\alpha$ is diagonalizable over
$\mathbb{Q}(t)(\alpha) \leq \mathbb{Q}((t^{-1}))$. That is to say
that there is some $g \in \sll$ such that $g \alpha g ^{-1}$ is
diagonal.

Because every element of $\mathbb{Z}[t][\alpha]$ is a linear
combination of powers of $\alpha$, we have that $g(
\mathbb{Z}[t][\alpha]) g^{-1}$ is a set of diagonal matrices. In
particular, we have proved part (ii) of Proposition~\ref{p:tori}.

\end{subsection}

\begin{subsection}{$A$ has trivial stabilizers}

To prove part (iii) of Proposition~\ref{p:tori} we begin with the
following

\begin{lemma}\label{l:alg} If $\Gamma \leq \slr$ is bounded under the valuation
for $\mathbb{Q}((t^{-1}))$, then the eigenvalues for any $\gamma
\in \Gamma$ lie in $\overline{\mathbb{Q}}$. \end{lemma}

\begin{proof} We let $X$ be the Euclidean building for $\sll$.
By assumption, $\Gamma  z =z$ for some $z \in X$.

Let $x_0 \in X$ be the vertex stabilized by $\slt$. We denote a
diagonal matrix in $\mathbf{GL_n}(\mathbb{Q}((t^{-1})))$ with
entries $s_1, s_2,..., s_n \in \mathbb{Q}((t^{-1}))^\times$ by
$D(s_1,s_2,...,s_n)$, and we let $\mathfrak{S} \subseteq X$ be the
sector based at $x_0$ and containing vertices of the form
$D(t^{m_1}, t^{m_2},...,t^{m_n})x_0$ where each $m_i \in
\mathbb{Z}$ and $m_1 \geq m_2 \geq ... \geq m_n$.

The sector $\mathfrak{S}$ is a fundamental domain for the action
of $\slr$ on $X$ \cite{So} which implies that there is some $h\in
\slr$ with $hz \in \mathfrak{S}$.

Clearly we have $(h\Gamma h^{-1} ) hz =hz$, and since eigenvalues
of $h\Gamma h^{-1}$ are the same as those for $\Gamma$, we may
assume that $\Gamma $ fixes a vertex $z \in \mathfrak{S}$.

Fix $m_1,...,m_n \in \mathbb{Z}$ with $m_1 \geq ... \geq m_n \geq
0$ and such that $z = D(t^{m_1},...,t^{m_n})x_0$. Without loss of
generality, there is a partition of $n$
--- say $\{k_1,...,k_\ell\}$ --- such that
$$\{m_1,...,m_n\}= \{ q_1,...,q_1, \; q_2,...,q_2, \;..., \; q _\ell ,...
q_\ell\}$$ where each $q_i$ occurs exactly $k_i$ times and
$$q_1>q_2>...>q_\ell$$

We have that $D(t^{m_1},...,t^{m_n})^{-1}\Gamma
D(t^{m_1},...,t^{m_n}) x_0 = x_0$. That gives us,
$D(t^{m_1},...,t^{m_n})^{-1} \Gamma D(t^{m_1},...,t^{m_n}) \subset
\slt$. Furthermore, a trivial calculation of resulting valuation
restrictions for the entries of \newline $D(t^{m_1},...,t^{m_n})
\slt D(t^{m_1},...,t^{m_n})^{-1}$ shows that $ \Gamma $ is
contained in a subgroup of $\sll$ that is isomorphic to
$$\prod_{i=1}^\ell \mathbf{SL_{k_i}}(\mathbb{Q})\ltimes U$$ where
$U \leq \sll$ is a group of upper-triangular unipotent matrices.

The lemma is proved.

\end{proof}

Our proof of Proposition~\ref{p:tori} will conclude by proving

\begin{lemma} No nontrivial element of $A$
 fixes a point in the Euclidean building for $\sll$.
 \end{lemma}

\begin{proof} Suppose $a \in A$ fixes a point in the building. We will show
that $a=1$.
Let $F(x)\in \mathbb{Z}[t][x]$ be the characteristic polynomial
for $a \in \slp$. Then
$$F(x)=\pm \prod_{i=1}^n(x - \beta _i)$$ where each $\beta _i \in
\mathbb{Q}((t^{-1}))$ is an eigenvalue of $a$. By the previous
lemma, each $\beta _i \in \overline{\mathbb{Q}}$. Hence, each
$\beta _i \in \mathbb{Q} = \overline{\mathbb{Q}} \cap
\mathbb{Q}((t^{-1}))$. It follows that $F(x) \in \mathbb{Z}[x]$ so
that each $\beta _i$ is an algebraic integer contained in
$\mathbb{Q}$. We conclude that each $\beta _i $ is contained in
$\mathbb{Z}$.

Recall, that $a$ has determinate $1$, and that the determinate of
$a$ can be expressed as $\prod _{i=1}^n \beta _i$. Hence, each
$\beta _i$ is a unit in $\mathbb{Z}$, so each eigenvalue $\beta_i
= \pm 1$. It follows -- by the diagonalizability of $a$ -- that
$a$ is a finite order element of $A \cong \mathbb{Z}^{n-1}$. That
is, $a=1$.

\end{proof}

We have completed our proof of Proposition~\ref{p:tori}.

\end{subsection}

\begin{section}{Body of the proof}\label{s:body}

Let $P \leq \sll$ be the subgroup where each of the first $n-1$
entries along the bottom row equal $0$. Let $R_u(P) \leq P$ be the
subgroup of elements that contain a $(n-1)\times(n-1)$ copy of the
identity matrix in the upper left corner. Thus $R_u(P) \cong
\mathbb{Q}((t^{-1}))^{n-1}$ with the operation of vector addition.

Let $L \leq P $ be the copy of
$\mathbf{SL_{n-1}}(\mathbb{Q}((t^{-1})))$ in the upper left corner
of $\sll$. We apply Proposition~\ref{p:tori} to $L$ (notice that
the $n$ in the proposition is now an $n-1$) to derive a subgroup
$A \leq L$ that is isomorphic to $\mathbb{Z}^{n-2}$. By the same
proposition, there is a matrix $g \in L$ such that $gAg^{-1}$ is
diagonal.

Let $b \in \sll$ be the diagonal matrix given in the notation from
the proofs of Lemmas~\ref{l:stab} and~\ref{l:alg} as
$D(t,t,...,t,t^{-(n-1)})$. Note that $b \in P$ commutes with $L$,
and therefore, with $A$. Thus the Zariski closure of the group
generated by $b$ and $A$ determines an apartment in $X$, namely
$g^{-1}\mathcal{A}$ where $\mathcal{A}$ is the apartment
corresponding to the diagonal subgroup of $\sll$.

\begin{subsection}{Actions on $g^{-1}\mathcal{A}$.}

If $x_* \in g^{-1}\mathcal{A}$, then it follows from
Proposition~\ref{p:tori} that the convex hull of the orbit of
$x_*$ under $A$ is an $(n-2)$-dimensional affine space that we
will name $V_{x_*}$. Furthermore, the orbit $Ax_*$ forms a lattice
in the space $V_{x_*}$.

We let $g^{-1}\mathcal{A}(\infty)$ be the visual boundary of
$g^{-1}\mathcal{A}$ in the Tits boundary of $X$. The visual image
of $V_{x_*}$ is clearly an equatorial sphere in
$g^{-1}\mathcal{A}(\infty)$. Precisely, we let $P^-$ be the
transpose of $P$. Then $P$ and $P^-$ are opposite vertices in
$g^{-1}\mathcal{A}(\infty)$. It follows that there is a unique
sphere in $g^{-1}\mathcal{A}(\infty)$ that is realized by all
points equidistant to $P$ and $P^-$. We call this sphere
$S_{P,P^-}$.

\begin{lemma}\label{l:eq} The visual boundary of $V_{x_*}$ equals
$S_{P,P^-}$.

\end{lemma}

\begin{proof}

Since $g \in P\cap P^-$, it suffices to prove that $gV_{x_*}$ is
the sphere in the boundary of $\mathcal{A}$ that is determined by
the vertices $P$ and $P^-$.

Note that $gV_{x_*}$ is a finite Hausdorff distance from any orbit
of a point in $\mathcal{A}$ under the action of the diagonal
subgroup of $L$. The result follows by observing that the inverse
transpose map on $\sll$ stabilizes diagonal matrices while
interchanging $P$ and $P^-$.

\end{proof}

We let $R_1,R_2,...,R_{n-1}$ be the standard root subgroups of
$R_u(P)$. Recall that associated to each $R_i$ there is a closed
geodesic hemisphere $H_i \subseteq \mathcal{A}(\infty)$ such that
any nontrivial element of $R_i$ fixes $H_i$ pointwise and
translates any point in the open hemisphere $\mathcal{A}(\infty) -
H_i$ outside of $\mathcal{A}(\infty)$. Note that $\partial H_i$ is
a codimension $1$ geodesic sphere in $\mathcal{A}(\infty)$.

We let $M \subseteq g^{-1}\mathcal{A}(\infty)$ be the union of
chambers in $g^{-1}\mathcal{A}(\infty)$ that contain the vertex
$P$. There is also an equivalent geometric description of $M$:

\begin{lemma}\label{l:simplex}
The union of chambers $M \subseteq g^{-1}\mathcal{A}(\infty)$ can
be realized as an $(n-2)$-simplex. Furthermore, $$M =
\bigcap_{i=1}^{n-1}g^{-1}H_i$$ and, when $M$ is realized as a
single simplex, each of the $n-1$ faces of $M$ is contained in a
unique equatorial sphere $g^{-1}\partial H_i = \partial
g^{-1}H_i$.
\end{lemma}

\begin{proof}
Let $M' \subseteq \mathcal{A}(\infty)$ be the union of chambers in
$\mathcal{A}(\infty)$ containing the vertex $P$. Since
$M=g^{-1}M'$, it suffices to prove that $M'$ is an $(n-2)$-simplex
with $M'=\cap_{i=1}^{n-1}H_i$ and with each face of $M'$ contained
in a unique $\partial H_i$.

For any nonempty, proper subset $I \subseteq \{1,2,...,n\}$, we
let $V_I$ be the $|I|$-dimensional vector subspace of
$\mathbb{Q}((t^{-1}))^n$ spanned by the coordinates given by $I$,
and we let $P_I$ be the stabilizer of $V_I$ in $\sll$. For
example, $P=P_{\{1,2,...,n-1\}}$.

Recall that the vertices of $\mathcal{A}(\infty)$ are given by the
parabolic groups $P_I$, that edges connect $P_I$ and $P_{I'}$
exactly when $I \subseteq I'$ or $I' \subseteq I$, and that the
remaining simplicial description of $\mathcal{A}(\infty)$ is given
by the condition that $\mathcal{A}(\infty)$ is a flag complex.

We let $\mathcal{V}$ be the set of vertices in
$\mathcal{A}(\infty)$ of the form $P_J$ where $\emptyset \neq J
\subseteq \{1,2,...,n-1\}$. Note that $M'$ is exactly the set of
vertices $\mathcal{V}$ together with the simplices described by
the incidence relations inherited from $\mathcal{A}(\infty)$.
Thus, $M'$ is easily seen to be isomorphic to a barycentric
subdivision of an abstract $(n-2)$-simplex. Indeed, if
$\overline{M'}$ is the abstract simplex on vertices
$P_{\{1\}},P_{\{2\}},...,P_{\{n-1\}}$, then a simplex of dimension
$k$ in $\overline{M'}$ corresponds to a unique $P_J \in
\mathcal{V}$ with $|J|=k+1$. So we  have that $M'$ can be
topologically realized as an $(n-2)$-simplex.

Let $F_i$ be a face of the simplex $\overline{M'}$. Then there is
some $1 \leq i \leq n-1$ such that the set of vertices of $F_i$ is
exactly
\mbox{$\{P_{\{1\}},P_{\{2\}},...,P_{\{n-1\}}\}-P_{\{i\}}$}.

Note that $R_i V_I =V_I$ exactly when $n \in I$ implies $i \in I$.
It follows that $R_i$ fixes $M'$ pointwise, and thus $M' \subseteq
H_i$ for all $1 \leq i \leq n-1$. Furthermore, if $P_I \in H_i$
for all $1 \leq i \leq n-1$, then $R_i P_I = P_I$ for all $i$ so
that $n \in I$ implies $i \in I$ for all $1 \leq i \leq n-1$. As
$I$ must be a proper subset of $\{1,2,...,n\}$, we have $P_I \in
\mathcal{V}$, so that $M'=\cap_{i=1}^{n-1}H_i$.

All that remains to be verified for this lemma is that $F_i
\subseteq \partial H_i$. For this fact, recall that $F_i$ is
comprised of $(n-3)$-simplices in $\mathcal{A}(\infty)$ whose
vertices are given by $P_J$ where $J \subseteq
\{1,2,...,n-1\}-\{i\}$. Hence, if $\sigma \subseteq
\mathcal{A}(\infty)$ is an $(n-3)$ simplex of
$\mathcal{A}(\infty)$ with $ \sigma \subseteq F_i$, then $\sigma $
is a face of exactly $2$ chambers in $\mathcal{A}(\infty)$:
$\mathfrak{C}_P$ and $\mathfrak{C}_{P_{J'}}$ where
$\mathfrak{C}_P$ contains $P$ and thus $\mathfrak{C}_P \subseteq
M'$, and $\mathfrak{C}_{P_{J'}}$ contains $P_{J'}$ where
$J'=\{1,2,...,n\}-\{i\}$ and thus $\mathfrak{C}_{P_{J'}}
\nsubseteq M'$. Furthermore, $\sigma = \mathfrak{C}_P \cap
\mathfrak{C}_{P_{J'}}$.

Since $R_i V_{J'}\neq V_{J'}$, it follows that
$\mathfrak{C}_{P_{J'}}$ is not fixed by $R_i$. Since
$\mathfrak{C}_{P_J}$ is fixed by $R_i$ we have that $\sigma =
\mathfrak{C}_P \cap \mathfrak{C}_{P_{J'}} \subseteq \partial H_i$.
 Therefore, $F_i \subseteq
\partial H_i$.
\end{proof}

For any vertex $y \in X$, we let $C_y \subseteq X$ be the union of
sectors based at $y$ and limiting to a chamber in $M$. Thus, $C_y$
is a cone. Note also that because any chamber in
$g^{-1}\mathcal{A}(\infty)$ has diameter less than $\pi /2$, it
follows that $M \cap S_{P,P^-}=\emptyset$. Therefore, if we choose
$x _* , y \in g^{-1}\mathcal{A}$ such that $x_*$ is  closer to $P$
than $y$, then $C_y \subseteq g^{-1}\mathcal{A}$ and $V_{x_*} \cap
C_y$ is a simplex of dimension $n-2$.

We will set on a fixed choice of $y$ before $x_*$, and we will
choose $y$ to satisfy the below

\begin{lemma}\label{l:cone} There is some $y \in g^{-1}\mathcal{A}$ such that
 the $\mathbb{Q}[[t^{-1}]]$-points of $R_u(P)$ fix $C_y$ pointwise.
\end{lemma}

\begin{proof} Let $x_0$ be the
 point in $X$ stabilized by $\slt$. Recall that $R_u(P)M=M$ so
 that the $\mathbb{Q}[[t^{-1}]]$-points of $R_u(P)$ fix $C_{x_0}$
 pointwise.

 Because $M \subseteq g^{-1}\mathcal{A}(\infty)$, there is a $y
 \in  C_{x_0} \cap g^{-1}\mathcal{A}$. Any such $y$ satisfies the
 lemma.

\end{proof}

Choose $e$ such that with $x_*=e$ as above and with $y$ as in
Lemma~\ref{l:cone}, there exists a fundamental domain $D_e$ for
the action of $A$ on $V_e$ that is contained in $ C_y$. The choice
of $e$ can be made by travelling arbitrarily far from $y$ along a
geodesic ray in $g^{-1}\mathcal{A}$ that limits to $P$.

By the choice of $D_e$ we have that
$$AD_e=V_e$$ and that the
$\mathbb{Q}([[t^{-1}]])$-points of $R_u(P)$ fix $D_e$.

\end{subsection}

\begin{subsection}{The filtration}

We let $$X_0=\slp D_e$$ and for any $i \in \mathbb{N}$ we choose
an $\slp$-invariant and cocompact space $X_i \subseteq X$ somewhat
arbitrarily to satisfy the inclusions
$$X_0 \subseteq X_1 \subseteq X_2 \subseteq ... \subseteq
\cup_{i=1}^\infty X_i =X$$

In our present context, Brown's criterion takes on the following
form \cite{Brown filtration}

\begin{bc} By Lemma~\ref{l:stab}, the group $\slp$ is not of type
$FP_{n-1}$ if for any $i \in \mathbb{N}$, there exists some class
in the homology group $\widetilde{\text{H}}_{n-2} (X_0 \,,\,
\mathbb{Z})$ which is nonzero in $\widetilde{\text{H}}_{n-2} (X_i
\,,\, \mathbb{Z})$.
\end{bc}

\end{subsection}

\begin{subsection}{Translation to $P$ moves away from filtration sets}

The following is essentially Mahler's compactness criterion.

\begin{lemma}\label{l:ik} Given any $i\in\mathbb{N}$, there is some $k \in
\mathbb{N}$ such that $b^ke\notin X_i$. \end{lemma}

\begin{proof} The lemma follows from showing that the sequence
$$\{\slp b^ke \}_k \subseteq \slp \backslash X$$ is unbounded.

Since stabilizers of points in $X$ are bounded subgroups of
$\sll$, the claim above follows from showing that the sequence
$$\{\slp b^k\}_k \subseteq \slp \backslash \sll$$ is unbounded.

But bounded sets in $\slp \backslash \sll$ do not contain
sequences of elements $\{\slp g_\ell\}_\ell$ such that $1 \in
\overline{g^{-1}_\ell ( \slp -\{1\} )g_\ell}$. And clearly $b^k$'s
contract some root groups to $1$. Thus none of the sequences above
is bounded. \end{proof}

\end{subsection}

\begin{subsection}{Applying Brown's criterion} As is described by Brown's criterion,
we will prove Theorem~\ref{t:main}
by fixing $X_i$ and finding an $(n-2)$-cycle in $X_0$ that
is nontrivial in the homology of $X_i$.

Recall that we denote the standard root subgroups of $R_u(P)$ by
$R_1,...,R_{n-1}$. Each group $g^{-1}R_jg$ determines a family of
parallel walls in $g^{-1}\mathcal{A}$. By Lemma~\ref{l:simplex},
each face of the cone $C_y$ is contained in a wall of one of these
families.

Choose $r_j \in g^{-1}R_jg$ for all $j$ such that $b^ke$ is
contained in the wall determined by $r_j$ where $k$ is determined
by $i$ as in Lemma~\ref{l:ik}. In particular, $r_j b^ke=b^ke$.

The intersection of the fixed point sets in $g^{-1} \mathcal{A}$
of the elements $r_1,...,r_{n-1}$ determine a cone that we name
$Z$. Note that $Z$ is contained in -- and is a finite Hausdorff
distance from -- the cone $C_y$.

Let $Z^-\subseteq g^{-1}\mathcal{A}$ be the closure of the set of
points in $g^{-1}\mathcal{A}$ that are fixed by none of the $r_j$.
The set $Z^-$ is a cone based at $b^ke$, containing $y$, and
asymptotically containing the vertex $P^-$.

As the walls of $Z^-$ are parallel to those of $Z$ -- and hence of
$C_y$, we have that $Z ^- \cap V_e$ is an $(n-2)$-dimensional
simplex. We will name this simplex $\sigma$.

The component of $Z^- - V_e$ that contains $b^k e$ is an
$(n-1)$-simplex that has $\sigma $ as a face. Call this $(n-1)$
simplex $Y$.

For any $\ell \in \mathbb{N}$, there are exactly $2^{n-1}$
possible subsets of the set $\{r_1^\ell,...,r_{n-1}^\ell\}$. For
each such subset $S_\ell$, we let
$$Y_{S_\ell} = (\prod _{g \in S_\ell} g)Y$$ and
$$ \sigma _{S_\ell} = (\prod_{g\in S _\ell} g)\sigma$$
Notice that the product of group elements in the equations above
are well-defined regardless of the order of the multiplication
since $R_u(P)$ is abelian. In the degenerate cases, $\prod_{g\in
\emptyset } g =1$, so $Y_{\emptyset } = Y$ and $\sigma _
{\emptyset } = \sigma$.

For any $\ell \in \mathbb{N}$, we let $Y_\ell = \cup _{S_\ell}
Y_{S_\ell}$. Because the wall in $g^{-1}\mathcal{A}$ determined by
$r_j^{\ell}$ is the same as the wall determined by $r_j$, the
space $Y_\ell$ is a closed ball containing $b^ke$ whose boundary
sphere is $\cup _{S _\ell} \sigma _{S _\ell}$. Indeed the
simplicial decomposition of $Y_\ell$ described above is isomorphic
to the simplicial decomposition of the unit ball in
$\mathbb{R}^{n-1}$ that is given by the $n-1$ hyperplanes defined
by setting a coordinate equal to $0$.

Let $\omega _\ell = \cup _{S _\ell} \sigma _{S_ \ell}$. Thus
$\omega _\ell = \partial Y_\ell$. Furthermore, the building $X$ is
$(n-1)$-dimensional and contractible, so any $(n-1)$-chain with
boundary equal to $\omega _\ell$ must contain $Y_\ell$ and thus
$b^k e$. That is for all $\ell \in \mathbb{N}$
$$[\omega _\ell] \neq 0 \in \widetilde{H}_{n-2}(X - b^ke \, , \, \mathbb{Z})$$
If we can show that $\omega _\ell \subseteq X_0$ for some choice
of $\ell$, then we will have proved our main theorem by
application of Brown's criterion since we would have
$$[\omega _\ell] \neq 0 \in \widetilde{H}_{n-2}(X_i \, , \, \mathbb{Z})$$ by
Lemma~\ref{l:ik}.

\begin{lemma} There exists some $\ell \in \mathbb{N}$ such that $\omega _\ell \subseteq X_0$.

\end{lemma}

\begin{proof}
For any $u \in R_u(P)$ there is a decomposition $u=u'u''$ where
the entries of $u' \in R_u(P)$ are contained in $\mathbb{Q}[t]$
and the entries of $u'' \in R_u(P)$ are contained in
$\mathbb{Q}[[t^{-1}]]$.

For any $a \in A$ and $u \in R_u(P)$ there is a power $\ell (a,u)
\in \mathbb{N}$ such that
$$(a^{-1}u^{\ell(a,u)}a)' =((a^{-1}ua)')^{\ell(a,u)} \in \slp$$
(For the above equality recall that $A \leq L$ normalizes $R_u(P)$
and the group operation on $R_u(P)$ is vector addition.)

There are only finitely many $a \in A$ such that $aD_e \cap \sigma
\neq \emptyset$ (or equivalently, such that $aD_e \cap Z^- \neq
\emptyset$). Call this finite set $\mathcal{D}\subseteq A$.

At this point we fix $$\ell = \prod_ {a \in\mathcal{D}} \prod _{i=1}
^{n-1}\ell(a,r_i)$$ Thus,
$$[a^{-1} (\prod_{g \in S _\ell }g) a ]' \in  \slp$$
for any $a \in\mathcal{D}$ and any $S _\ell \subseteq \{r_i ^\ell
\}_{i=1}^{n-1}$.

Because $\omega _\ell = \cup _{S_ \ell} \sigma _{S _\ell}$ and $
\sigma _{S_\ell} = (\prod_{g\in S _\ell} g)\sigma =  (\prod_{g\in
S_\ell} g) (AD_e \cap Z^-)$, we can finish our proof of this lemma
by showing
$$\big(\prod_{g \in S_\ell}g\big)aD_e \subseteq X_0$$ for each $a\in\mathcal{D} \subseteq A \leq \slp$ and
each $S _\ell \subseteq \{r_i ^\ell \}_{i=1}^{n-1}$. For this,
recall that the $\mathbb{Q}[[t^{-1}]]$-points of $R_u(P)$ fix
$D_e$ and thus

\begin{align*} \big( \prod_{g \in S_\ell}g \big) a D_e & =
 a [a^{-1} \big( \prod_{g \in S_\ell}g \big) a ]D_e
  \\ & =a[a^{-1} \big( \prod_{g \in S_\ell}g \big) a ]'[a^{-1} \big( \prod_{g \in S_\ell}g \big) a ]''D_e \\
& =a [a^{-1} \big( \prod_{g \in S_\ell}g \big) a ]' D_e \\
& \subseteq \slp D_e \\ & = X_0
\end{align*}

\end{proof}

\end{subsection}

\end{section}

\bigskip \noindent \textbf{Authors email addresses:} \newline \noindent kb2ue@virginia.edu
\newline \noindent amir.mohammadi@yale.edu \newline \noindent
wortman@math.utah.edu


\begin{thebibliography}{KKLL}

\bibitem[Bo-Se]{B-S} A. Borel, J-P. Serre, \emph{Corners and arithmetic groups.} Commentarii
Mathematici Helvetici {\bf 48} (1973) p. 436 Ð 491.


\bibitem[Br]{Brown filtration} Brown, K.,
\emph{Finiteness properties of groups.} J. Pure Appl. Algebra {\bf
44} (1987), 45-75.


\bibitem[Bu-Wo 1]{BW1} Bux, K.-U., and Wortman, K., \emph{Finiteness
 properties of arithmetic groups over function fields.}
 Invent. Math. \textbf{167} (2007), 355-378.

\bibitem[Bu-Wo 2]{BW2} Bux, K.-U., and Wortman, K., \emph{Geometric proof that
$\mathbf{SL_2}$$($$\mathbf{Z}$$[t,t^{-1}])$ is not finitely
presented.} Algebr. Geom. Topol.  \textbf{6} (2006), 839-852.

\bibitem[Kr-Mc]{K-M} Krsti\'{c}, S., and McCool, J.,
\emph{Presenting ${\rm GL}_n(k\langle T \rangle )$.} J. Pure Appl.
Algebra {\bf 141} (1999), 175-183.

\bibitem[Na]{Nagao} Nagao, H., \emph{On} $\text{GL}(2,K[X])$.
J. Inst. Polytech. Osaka City Univ. Ser. A {\bf 10} (1959),
117-121.




\bibitem[Su]{Suslin} Suslin, A. A.
\emph{The structure of the special linear group over rings of polynomials.}
(Russian)
Izv. Akad. Nauk SSSR Ser. Mat. {\bf 41} (1977), no. 2, 235--252, 477.

\bibitem[So]{So} Soul$\acute{e}$, C., \emph{Chevalley groups over polynomial rings.} Homological Group Theory, LMS \textbf{36} (1977), 359-367.

\end{thebibliography}
\end{document}